\documentclass[10pt]{amsart}
\usepackage{amssymb,latexsym,amsmath,amsthm,enumitem,mathrsfs,geometry}
\usepackage{fullpage}
\usepackage{fancyhdr}

\usepackage{comment}
\usepackage{mathrsfs}
\usepackage{hyperref}
\usepackage{lineno}
\usepackage{diagbox}
\usepackage{array}
\usepackage{tikz}
\usetikzlibrary{arrows}
\usetikzlibrary{positioning}
\usepackage{float}

\newtheorem{thm}{Theorem}[section]
\newtheorem*{thm*}{Theorem}

\newtheorem{lemma}[thm]{Lemma}

\newcommand{\beq}{\begin{equation}}
\newcommand{\eeq}{\end{equation}}

\newtheorem*{remark}{Remark}

\usepackage[titletoc,title]{appendix}

\def\A{\mathbb{A}}
\newcommand{\Q}{\mathbb{Q}}
\newcommand{\R}{\mathbb{R}}

\newcommand{\Z}{\mathbb{Z}}

\newcommand{\PKN}{\Pi_{\underline{k}}(\mathfrak{n})}
\newcommand{\Nr}{\mathrm{N}}

\newcommand{\pfrak}{\mathfrak{p}}

\newcommand{\kommentar}[1]{}

\newtheorem*{remark*}{Remark}

\newtheorem{proposition}{Proposition}[section]

\definecolor{pink}{rgb}{1,.2,.6}
\definecolor{orange}{rgb}{0.7,0.3,0}
\definecolor{blue}{rgb}{.2,.6,.75}
\definecolor{green}{rgb}{.4,.7,.4}
\definecolor{purple}{RGB}{127,0,255}

\begin{document}
\numberwithin{equation}{section}

\title{A central limit theorem for Hilbert Modular Forms }

\author[Das]{Jishu Das}
\address{Indian Institute of Science Education and Research Pune, Dr. Homi Bhabha Road, Pune 411008}
\email{jishu.das@students.iiserpune.ac.in}

\author[Prabhu]{Neha Prabhu}
\address{ Sai University, One Hub Road, Paiyanur, Tamil Nadu 603104}
\email{neha.prabhu@acads.iiserpune.ac.in}

\thanks{}

\date{\today}

\begin{abstract} For a prime ideal $\pfrak$ in a totally real number field $L$ with the adele ring $\mathbb{A}$, we study the distribution of angles $\theta_\pi(\pfrak)$ coming from Satake parameters corresponding to unramified $\pi_\pfrak$ where $\pi_\pfrak$ comes from a global  $\pi$ ranging over  a certain finite set $\PKN$ of cuspidal automorphic representations of GL$_2(\mathbb{A})$ with trivial central character.   For such a representation $\pi$, it is known that the angles $\theta_\pi(\pfrak)$ follow the Sato-Tate distribution. Fixing an interval $I\subseteq [0,\pi]$, we prove a central limit theorem for the number of angles $\theta_\pi(\pfrak)$ that lie in $I$, as $\Nr(\pfrak)\to\infty$. The result assumes $\mathfrak{n}$ to be a squarefree integral ideal, and that the components in the weight vector $\underline{k}$ grow suitably fast as a function of $x$.
\end{abstract}
\maketitle 
\section{Introduction}

The statistics of eigenvalues of Hecke operators have been a topic of interest for a few decades. More recently, following the series of papers which settled the Sato-Tate conjecture in various settings, such as \cite{Taylor2008}, \cite{BLGG} \cite{BLGG}, the study of error terms in these theorems has received significant attention, see \cite{Prabhu-Sinha, Baier-Prabhu, Thorner21, HIJT} for example. This article investigates the statistics of the error term in the Sato-Tate theorem for Hilbert Modular Forms, building on a similar study conducted in \cite{Prabhu-Sinha}, which we briefly describe. Let $\mathcal{F}_{N,k}$ be the set of normalized non-CM cusp forms of weight $k$ and level $N$ that are also eigenforms for Hecke operators $T_n$ acting on spaces of cusp forms $S(N,k)$. In particular, any $f\in \mathcal{F}_{N,k}$ has a Fourier expansion $$f(z)= \sum\limits_{n=1}^\infty n^{\frac{k-1}{2}} a_f(n)q^n$$ with $a_f(1)= 1$ and $q= e^{2\pi i z}$. By the work of Deligne, it is known that for $p$ prime with $\gcd(p, N)=1$, the sequence $a_f(p)$ lies in $[-2,2]$ and the Sato-Tate theorem reveals the distribution of this sequence. That is, for a fixed $f\in \mathcal{F}_{N,k}$, if we set $a_f(p) = 2\cos\theta_f(p)$, and fix an interval $I \subseteq [0,\pi]$ and let $$N_I(f,x) := \# \{ p\leq x: \gcd(p,N)=1, \theta_f(p)\in I\}, $$ then $$ \lim\limits_{x\to\infty} \frac{N_I(f,x)}{\pi(x)} = \int_I \mu_\infty(t) ~dt. $$

Here, $\pi(x)$ denotes the number of primes not exceeding $x$, and $\mu_\infty(t) = \frac{2}{\pi} \sin^2 t$ is the function associated with the Sato-Tate measure of the interval $I$. In \cite{Prabhu-Sinha}, it was shown that under suitable growth conditions of the weight $k= k(x)$, the error term in the above asymptotic statement exhibits a Gaussian distribution when one averages over $f$ in $\mathcal{F}_{N,k}$. In this article, we prove the analogous result in the Hilbert modular setting. 

Let $L$ denote a totally real field of degree $d$ over $\Q$ and $\mathcal{O}$ be its ring of integers.  Let $\nu=\nu_{\mathfrak{p}}$ be the discrete valuation associated with the prime ideal $\mathfrak{p}$ in $\mathcal{O}.$ The local field $L_\nu$ is the completion of $L$ with respect to the topology induced by $\nu$, and $\mathcal{O}_\nu$ denotes the ring of integers in the local field $L_\nu.$  Let $\A$ denote the Adele ring of  $\Q$ and $\A_{f}$ denote the finite adeles.  Let $G$ denote the algebraic group that is the Weil restriction of scalars of $\text{GL}_{2/L}$ from $L$ to $\Q$. Let $G_f=G(\A_{f})$ and $G_\infty=G(\R)$. For an integral ideal $\mathfrak{n}$, let  $K_0(\mathfrak{n})$ be the congruence subgroup of level $\mathfrak{n}.$  If the prime factorisation of $\mathfrak{n}$ is given by $\mathfrak{n}=\mathfrak{q}_1^{a_1}\,\dots\, \mathfrak{q}_r^{a_r},$ then the congruence subgroup $K_0(\mathfrak{n})$ is defined as 
$$K_0(\mathfrak{n})=\Bigg\{ \begin{pmatrix}
a & b \\
c & d 
\end{pmatrix} \in \prod_\nu GL_2(\mathcal{O}_\nu)  \,\Big|\,   
c_{\mathfrak{q}_i}\equiv 0 \, \text{mod} \,\mathfrak{q}_i^{a_i} \, \text{for all} \, i=1,\dots,r
\Bigg\}.
$$

Let the weight be given by the  $d$-tuple $\underline{k}=(k_1,\dots,k_d)$ where $k_i$ is an even integer and  $k_i\geq 4$ for all $i=1,\dots , d.$
Let $\mathcal{A}^{\text{cusp}}_{\underline{k}}(G(\Q) \backslash G(\A))$ denote the space of  cuspidal automorphic forms on
$G(\A)$ with trivial central character (see for instance \cite[Section 4]{BorelJacquet} for definitions). The group $G(\A) $ acts on $\mathcal{A}^{\text{cusp}}_{\underline{k}}(G(\Q) \backslash G(\A))$ by right translations. An irreducible representation $\pi$ of $G(\A)$ is called a cuspidal automorphic representation if it is isomorphic to a subrepresentation of $\mathcal{A}^{\text{cusp}}_{\underline{k}}(G(\Q) \backslash G(\A)).$
We have the following decomposition of $\pi=\pi_f\otimes \pi_\infty$ where $\pi_f$ and $\pi_\infty$ are representations of $G_f$ and $G_\infty$ respectively.

Let $\PKN$ be the set of cuspidal unitary automorphic representations $\pi$ with respect to $G$ such that 
\begin{itemize}
    \item $\pi_f$ has a $K_0(\mathfrak{n})$ fixed vector,
    \item $\pi_\infty=\otimes_{i=1}^d D_{k_i-1},$
    where $D_k$ is the discrete series representation of $\text{GL}_2(\R)$ with minimal $K$-type of weight $k+1.$ 
\end{itemize}
The set $\PKN$ is finite (see \eqref{size of PKN}). Let $\pi\in \PKN$  and consider $\mathfrak{p}$, a prime ideal in $L$ for which $\pi_\pfrak$  is unramified.  
The Satake parameter associated with $\pi_\pfrak$ is a conjugacy class 

$$\begin{pmatrix}
e^{i\pi \theta_\pi(\pfrak)} &  \\
 & e^{-i\pi \theta_\pi(\pfrak)}
\end{pmatrix} \in \text{SU(2)}/\sim$$ with $\theta_\pi(\pfrak)\in [0,1].$
For the classical setting of  $L=\Q$ and a classical eigenform $f,$ let $\pi(f) $ denote the automorphic representation associated with $f$. The angles $\theta_f(p)$ considered in \cite{Prabhu-Sinha} are exactly the $\theta_\pi(\pfrak) $ with $\pfrak=\langle p \rangle.$

 We set $$ \pi_L(x) = \# \{ \pfrak : \pfrak \text{ prime ideal in $L$ with } \pfrak\nmid \mathfrak{n}, \Nr(\pfrak)\leq x \}.$$
For an interval $I= [\alpha, \beta] \subseteq [0, \pi]$, the Sato-Tate theorem studies the distribution of 
\begin{align*}
    N_I(\pi, x) := \sum\limits_{\substack{\Nr(\pfrak)\leq x\\ \pfrak \nmid \mathfrak{n} }} \chi_I\left( \theta_\pi(\pfrak)\right)
\end{align*} and from the work of Barnet-Lamb, Gee and Geraghty \cite{BLGG}, it is known that
\begin{equation}\label{Sato-Tate for HMF}
    N_I(\pi, x) \sim \pi_L(x) \mu_{\infty}(I).
\end{equation}
We are interested in studying the statistics of the error term in this theorem. To be precise, let $\phi$ be a complex-valued function defined on $\PKN$. We denote the average 
$$\langle\phi(\pi)\rangle :=\frac{1}{\#\PKN} \sum_{\pi \in \PKN} \phi(\pi).$$ The task then, is to study the behaviour of the moments 
\begin{equation}\label{r-th-moments}
    \langle (N_I(\pi, x) -\pi_L(x) \mu_{\infty}(I))^r \rangle 
\end{equation}for each $r\in \mathbb{N}$.
We prove
\begin{thm}\label{main-thm}
    Consider a family $\PKN$ for fixed squarefree level $\mathfrak{n}$ and even weights $\underline{k} =\underline{k}(x)$ such that  $\frac{\sum_{i=1}^d \log k_i}{\sqrt{x}\log x}\to \infty$ as $x\to\infty$. Fix an interval $I \subseteq [0,\pi]$. Then for any continuous real-valued function $g$ on $\R$, we have
    $$ \lim\limits_{x\to\infty} g \left( \frac{N_I(\pi, x)- \pi_L(x)\mu_\infty(I)}{\sqrt{\pi_L(x) (\mu_\infty(I) - \mu_\infty(I)^2)}}\right) = \frac{1}{\sqrt{2\pi}}\int_{-\infty}^{\infty} g(t)e^{-\frac{t^2}{2}}~dt.$$
\end{thm}
The motivation for the above result comes from ideas in probability theory. For a given interval $I\subseteq [0,\pi]$ and a representation $\pi$ in $\PKN$, consider a prime ideal $\pfrak$ such that  $\pi_\pfrak$ is unramified. Note that since $\pfrak\nmid \mathfrak{n}$, choosing $\pi$ with a $K_0(\mathfrak{n})$ 
fixed vector gives us that the representations $\pi_\pfrak$ are unramified. Let
$$ X_{\pfrak, \pi} = \chi_I\left( \theta_\pi(\pfrak)\right). $$ Consider a non-archimedean valuation $\nu$ corresponding to $\pfrak$ and set
\begin{equation*}
        d\mu_\nu (\theta) = \frac{ \Nr(\pfrak) + 1}{\left( \Nr(\pfrak)^{\frac{1}{2}} + \Nr(\pfrak)^{-\frac{1}{2}}\right)^2 -4\cos^2\theta}d\mu_\infty(\theta),
    \end{equation*}
    where $ d\mu_\infty(\theta) = \frac{2}{\pi} \sin^2\theta ~d\theta.$
    From the work of Li \cite{Li}, we know that $$ \langle X_{\pfrak,\pi} \rangle \sim \mu_\nu(I)$$ if we assume appropriate growth conditions on $\underline{k}$ or $\mathfrak{n}$. More generally, Lau-Li-Wang \cite[Theorem 1.1]{Lau-Li-Wang} prove an effective joint distribution result for these angles. Using their result, we infer that if $I_1,\ldots, I_h$ are each intervals in $[0,\pi]$ and $\pfrak_1, \ldots, \pfrak_h$ are distinct prime ideals that are relatively prime to $\mathfrak{n},$ then
    $$ \langle X_{\pfrak_1,\pi}\cdots X_{\pfrak_h,\pi} \rangle \sim \prod_{j=1}^T \mu_{\nu_j}(I_j)$$ under suitable growth conditions on $\underline{k}$. Here, $\nu_j$ is the valuation corresponding to $\pfrak_j$. Thus, viewed as random variables, the $X_{\pfrak,\pi}$ have distributions that depend on $\pfrak$, but behave like independent random variables for large enough $\underline{k}$.
    
    It is then natural to investigate the distribution of the sum of random variables $\sum_{\Nr(\pfrak)\leq x} X_{\pfrak,\pi}$. Firstly, we view $\mu_\nu(I)$ to be $\mbox{E}[X_{\pfrak,\pi}]$. Then a guess for the variance would be $$ \text{Var}[X_{\pfrak,\pi}]= \mbox{E}[X_{\pfrak,\pi}^2]- \mbox{E}[X_{\pfrak,\pi}]^2= \mu_\nu(I) - \mu_\nu(I)^2.$$
 However, keeping \eqref{Sato-Tate for HMF} in mind, and observing that $d\mu_{\nu}(\theta) \to d\mu_{\infty}(\theta)$ as $\Nr(\pfrak)\to\infty$, the quantity  $\pi_L(x)\mu_\infty(I)$ is a reasonable expression for the expectation of the sum of random variables $\sum_{\Nr(\pfrak)\leq x} X_{\pfrak,\pi}$ for $x$ large enough. Similarly, the variance of the sum of random variables is heuristically $ \pi_L(x) (\mu_\infty(I)-\mu_\infty(I)^2)$ as $x\to\infty$. With these ideas in mind, we are motivated to find necessary conditions for a central limit theorem to hold. This also adds to the literature on central limit theorems for Hecke eigenvalues in various settings, for example \cite{Nagoshi}, \cite{Murty-Prabhu}, \cite{lau-Ng-Wang} and \cite{Fugleberg-Walji}. 

The proof of Theorem \ref{main-thm} adapts the underlying technique used in the proof of \cite[Theorem 1.1]{Prabhu-Sinha}, namely the method of moments applied to quantities approximating the moments \eqref{r-th-moments}, and the Eichler-Selberg trace formula. However, the proof of Theorem \ref{main-thm} is cleaner, and is achieved by two processes. First, we expand the approximating functions using a basis of Chebyshev polynomials. Secondly, we simplify the proof of the theorem pertaining to the calculation of higher moments. This proof is adapted from \cite{Sun-Wen-Zhang}, who worked out a simplification of the theorem on higher moments in \cite{Prabhu-Sinha}. This is achieved by expressing the main term of the analogous trace formula, given in \eqref{trace-formula-LLW}, as an integral involving Chebyshev polynomials (see Lemma \ref{alternate version of trace formula}). 

The structure of the paper is as follows. Section \ref{Preliminaries} covers the properties of Chebyshev polynomials and Beurling-Selberg polynomials needed to follow the proof of the theorem. Section \ref{first-moment-section} gives the details of the first-moment calculation, culminating in the average Sato-Tate result for Hilbert Modular Forms. After explaining the outline of the proof of Theorem \ref{main-thm} in Section \ref{Outline}, we give the details of the simplified proof of higher moments in Section \ref{higher-moments}. We conclude the article by stating a smooth version of Theorem \ref{higher-moments-theorem} that should hold, resulting in a smooth version of the central limit theorem with weaker growth conditions on $\underline{k}=\underline{k}(x)$. 

{\bf Acknowledgements} The first named author is supported by a Ph.D. fellowship from CSIR. The second named author's research is supported by a DST-INSPIRE Faculty Fellowship from the Government of India and acknowledges Savitribai Phule Pune University, where she was working when this work began. The authors are grateful to Baskar Balasubramanyam and Kaneenika Sinha for their valuable comments on a previous version of this article.

\section{Preliminaries}\label{Preliminaries}
\subsection{Chebyshev polynomials}

    The Chebyshev polynomials of the second kind, denoted by $U_n$, are defined by the recurrence relations
\begin{align*}
U_0(x)&=1\\
U_1(x)&=2x\\
U_{n}(x)&=2x\,U_{n-1}(x)-U_{n-2}(x) \qquad \text{for } n\geq 2.
\end{align*}
These polynomials form an orthonormal basis with respect to the Sato-Tate measure. More explicitly, the following holds. For non-negative integers $m ,n $
\begin{align}\label{orthogonality-rel'ns}
    \frac{2}{\pi} \int_0^\pi U_m(\cos\theta)U_n(\cos\theta) \sin^2\theta ~d\theta = \begin{cases}
        1 & \text{ if } m=n\\
        0 &\text{ otherwise.}
    \end{cases}
\end{align}
The following recursive relations also hold. For non-negative integers $m$ and $n$ with $m\geq n$,
\begin{equation}\label{recursive-rel'ns}
    U_m(x)U_n(x) = \sum\limits_{k=0}^n U_{m+n-2k}(x).
\end{equation}
\subsection{A trace formula}\label{sec-trace formula}
    We state the version of Arthur's trace formula proved in \cite{Lau-Li-Wang} that we will frequently use. This also appears in \cite{BB-KS}. 
    \begin{proposition}[\cite{BB-KS}, Proposition 18]\label{trace-formula}
        Let $\pfrak_1,\ldots, \pfrak_h$ be distinct primes coprime to squarefree $\mathfrak{n}$. Let $\underline{m}=(m_1,\ldots,m_h)$ be a tuple of non-negative integers, and $\mathfrak{a} = \pfrak_1^{m_1}\cdots \pfrak_h^{m_h}$. Then 

\begin{equation}\label{trace-formula-LLW}
    \sum_{\pi \in \PKN} \prod_{i=1}^h U_{m_i}(\cos\theta_\pi(\pfrak_i)) = C_L\Nr(\mathfrak{n})\delta_{2\mid \underline{m}} \prod_{i=1}^d \frac{k_i-1}{4\pi} \Nr(\mathfrak{a})^{-\frac{1}{2}} + O\left(\Nr(\mathfrak{n})^\epsilon\Nr(\mathfrak{a})^{\frac{3}{2}} \right).
\end{equation}
Here, 
\begin{itemize}
    \item $C_L$ is a constant that depends only on the number field $L$.
    \item The quantity $\delta_{2\mid \underline{m}}$ is equal to one if all the $m_i$ are even, and zero otherwise.
\end{itemize}
    \end{proposition}
    As an immediate consequence, we have 
    \begin{equation}\label{size of PKN}
        \# \PKN = C_L\Nr(\mathfrak{n})\prod_{i=1}^d \frac{k_i-1}{4\pi} + O\left(\Nr(\mathfrak{n})^\epsilon \right).
    \end{equation}
    \begin{remark}
        For simplicity, we stick to squarefree level $\mathfrak{n}$ keeping Proposition \ref{trace-formula} in mind. A similar result may be obtained by taking $\mathfrak{n}$ not necessarily squarefree, using [\cite{Lau-Li-Wang}, Theorem 6.3].  
    \end{remark}
We also state well-known results on sums of powers of prime ideal norms that will be useful. 
 \begin{lemma}\label{NP>=1}
     We have 
     \begin{equation}\label{sumnp1}
   \sum_{\Nr(\pfrak)\leq x} \frac{1}{\Nr(\pfrak)}=\log\log x+O_L(1) 
\end{equation} and 
     \begin{equation}\label{sumnp>1}
     \sum_{r\geq 2} \sum_{\Nr(\pfrak)\leq x} \frac{1}{\Nr(\pfrak)^r}=O_L(1).
  \end{equation}
 \end{lemma}
 \begin{proof}
 Equation \eqref{sumnp1} follows from Lemma 2.4 of \cite{MR}. For \eqref{sumnp>1}, recall that for a prime ideal $\mathfrak{p},$ the value of $\Nr(\pfrak)=p^t$ where $\langle p\rangle =\pfrak\cap \Z$ and $1\leq t\leq d. $  Let the ideal counting function be given by $$a_{l}=|\{   \mathfrak{m} \subset \mathcal{O}  : \mathfrak{m}  \,\text{ideal with } \Nr(\mathfrak{m})=l. \}|.$$
      Then, it is known that $a_l\leq \tau(l)^{d-1}$ (see \cite[equation (68)]{CN}), where $\tau$ is the usual divisor function.
    
Therefore,
\begin{align*}
 \sum_{r\geq 2} \sum_{\Nr(\pfrak)\leq x} \frac{1}{\Nr(\pfrak)^r}&=\sum_{r\geq 2}\sum_{p\leq x}\sum_{t=1}^d \frac{a_{p^t}}{p^{tr}}\leq \sum_{r\geq 2}\sum_{p\leq x}\sum_{t=1}^{d} \frac{(t+1)^{d-1}}{p^{tr}}
 \\
 &\leq \sum_{r\geq 2}\sum_{p\leq x} \frac{d(d+1)^{d-1}}{p^{r}}
 \ll_L \sum_{r\geq 2}\sum_{p\leq x} \frac{1}{p^{r}}=O_L(1).
\end{align*}
 \end{proof}    

\subsection{Beurling-Selberg polynomials}
The main underlying technique used in this article is to approximate the characteristic function by appropriate trigonometric polynomials of finite degree. This technique has been used widely to obtain fluctuations in error terms in equidistribution theorems, following the work of \cite{Faifman-Rudnick}. Here, we state the properties that we need, and refer the reader to a detailed exposition in \cite[Chapter 1]{Montgomery}. 
Let $J=[\alpha,\beta]\subseteq [-\frac{1}{2}, \frac{1}{2}]$ and $M\geq 1$ be an integer. There exist trigonometric polynomials $S_{J,M}^+(x)$ and $S_{J,M}^-(x)$ of degree not exceeding $M$ such that for all $x\in \R$, $$S_{J,M}^-(x) \leq \chi_J(x) \leq S_{J,M}^+(x),$$ where $\chi_J$ is the usual indicator function of the interval $J$. While the explicit definition of these polynomials can be found in \cite{Montgomery}, we will heavily use properties of the coefficients in their Fourier expansions for our calculations. Before we describe these properties, we fix some notation.
Let $e(t):= e^{2\pi it}.$ Then, the Fourier expansion of $\chi_I$ is given by 
$$ \chi_J(x) = \sum_{n\in \Z} \hat{\chi}_J(n)e(nx), \qquad \qquad \hat{\chi}_J(n) = \int_J e(-nt)~dt.$$
Note that 
\begin{equation}\label{chi-fourier coeff}
    \hat{\chi}_J(0) = \beta-\alpha, \qquad \hat{\chi}_J(n)= \frac{e(-n\alpha)-e(-n\beta)}{2\pi i n} \quad \text{ for } |n|\geq 1.
\end{equation}
The Beurling-Selberg polynomials $S_{J,M}^+$ and $S_{J,M}^-$ are good approximations of the indicator function and this is evident from the properties of its Fourier coefficients. For $0\leq |m| \leq M$, 
\begin{equation}\label{S-hat}
    \hat{S}_{J,M}^{\pm}(m) =   \hat{\chi}_J(m) + O\left( \frac{1}{M+1}\right)
\end{equation} and $\hat{S}_{J,M}^{\pm}(m) =0$ for $|m|>M$. Since we are interested in finer statistics of the sequence of angles $\{ \theta_\pi(\pfrak)\}$, which are equidistributed with respect to the measure $\frac{2}{\pi} \sin^2\theta$, it is useful to consider Fourier expansions of the polynomials using the orthonormal basis of Chebyshev polynomials $U_n(\cos\theta)$. The following lemma accomplishes this.

	\begin{lemma}[{\cite[Lemma 1.3]{RT2017}}] \label{lemma-RT}
		Let $I = [a, b] \subseteq[0,\pi]$, and let $M$ be a positive integer. There exist trigonometric polynomials $$F^{\pm}_{I,M}(\theta) = \sum\limits_{m=0}^M \hat{F}^{\pm}_{I,M}(n)U_m(\cos\theta)$$ such that for $0\leq \theta \leq \pi$, we have $$F^{-}_{I,M}(\theta)\leq \chi_I(\theta) \leq F^{+}_{I,M}(\theta). $$
			
	\end{lemma}

The key idea in the proof is to take $\alpha= \frac{a}{2\pi}$ and $\beta= \frac{b}{2\pi}$ so that we now have an interval in $[-\frac{1}{2}, \frac{1}{2}]$, and study the properties of $$F^{\pm}_{I,M}(\theta)= S^{\pm}_{J, M}\left(\frac{\theta}{2\pi} \right) + S^{\pm}_{J, M}\left(-\frac{\theta}{2\pi} \right).$$ 
The relation between the Fourier coefficients $\hat{S}_{J,M}^\pm(m)$ and $\hat{F}_{I,M}^\pm(m)$ is the following. For $0\leq m\leq M$, let 
\begin{equation}\label{mathcal-S-defn}
    \hat{\mathcal{S}}_{J,M}^\pm(m) = \hat{S}_{J,M}^\pm(m) + \hat{S}_{J,M}^\pm(-m).
\end{equation} Then
\begin{align}
    \hat{F}_{I,M}^\pm(m)&= \hat{\mathcal{S}}_{J,M}^\pm(m) -\hat{\mathcal{S}}_{J,M}^\pm(m+2).
\end{align}
Moreover, using \eqref{chi-fourier coeff} and \eqref{S-hat} in \eqref{mathcal-S-defn}, we get that for $0<m\leq M$,
\begin{align}\label{mathcal-S-explicit}
    \hat{\mathcal{S}}_{J,M}^\pm(m) &= \frac{ \sin(2\pi m\beta)-\sin(2\pi m\alpha)}{m\pi} +O\left(\frac{1}{M+1}\right)\\
  \notag  &= \frac{\sin(mb)-\sin(ma)}{m\pi} +O\left(\frac{1}{M+1}\right), \quad \text{ and }\\
    \hat{\mathcal{S}}_{J,M}^\pm(0) &= 2(\beta-\alpha) = \frac{a}{\pi} -\frac{b}{\pi}.
\end{align}
Henceforth, to simplify notation we will fix the interval $I= [a, b]$ throughout the rest of the article, and write $F_M^{\pm}$ and $\hat{F}_M^{\pm}$ without mentioning $I$ in the subscript. 

We now record an important result about the Fourier coefficients $\hat{F}_{M}^\pm(m)$.
\begin{proposition}
    \begin{align}
      \label{expected-value1}  \hat{F}_M^{\pm}(0) &= \mu_\infty(I) + O\left( \frac{1}{M+1} \right).\\
       \label{variance} \sum_{m=1}^M \hat{F}_M^{\pm}(m)^2  &=  \mu_\infty(I) - \mu_\infty(I)^2 + O\left(\frac{\log M}{M} \right).
    \end{align}
\end{proposition}
\begin{proof}
    Equation \eqref{expected-value1} follows easily from noting that 
    \begin{align*}
        \hat{F}_M^{\pm}(0) &= \hat{\mathcal{S}}_{J,M}^\pm(0) -\hat{\mathcal{S}}_{J,M}^\pm(2)\\
        &= \frac{a}{\pi} -\frac{b}{\pi} - \frac{\sin(2b)-\sin(2a)}{2\pi} +O\left(\frac{1}{M+1}\right)\\
        &= \frac{2}{\pi}\int_a^b \sin^2\theta~d\theta + O\left(\frac{1}{M+1}\right)
    \end{align*}
    Equation \eqref{variance} is less straightforward, and has been worked out in \cite[Proposition 3.6.1]{NPThesis}.
\end{proof}


\section{First Moment}\label{first-moment-section}
Throughout the rest of this article, whenever we write $ \Nr(\pfrak)\leq x$, in a sum, we will assume that the sum runs over prime ideals $\pfrak \nmid \mathfrak{n}$.

Approximating the quantity $N_I(\pi, x)$ above and below by the Beurling-Selberg polynomials, we have 
$$ \sum\limits_{\Nr(\pfrak)\leq x} F^{-}_{M}(\theta_\pi(\pfrak))\leq N_I(\pi, x) \leq \sum\limits_{\Nr(\pfrak)\leq x} F^{+}_{M}(\theta_\pi(\pfrak)).$$
Writing 
$$ \sum\limits_{\Nr(\pfrak)\leq x} F^{\pm}_{M}(\theta_\pi(\pfrak)) = \pi_L(x) \hat{F}^\pm_M(0) + \sum_{m=1}^M     \hat{F}^\pm_M(m)  \sum_{\Nr(\pfrak)\leq x} U_m(\cos (\theta_\pi(\pfrak))),$$
and using \eqref{expected-value1}, there exist constants $C$ and $D$ such that
\begin{align*}
    D\frac{\pi_L(x)}{M+1} + F^-(M,\pi)(x) \leq N_I(\pi, x) -\pi_L(x)\mu_\infty(I) \leq F^+(M, \pi)(x) + C\frac{\pi_L(x)}{M+1}
\end{align*} where 
\begin{equation}\label{F(M,pi)-defn}
    F^{\pm}(M,\pi)(x):=\sum_{m=1}^M \hat{F}_M^{\pm}(m) \sum_{\Nr(\pfrak)\leq x} U_m( \cos \theta_\pi(\pfrak)).
\end{equation}
 We compute the first moment by estimating
$$ \frac{1}{\#\PKN} \sum_{\pi\in \PKN} F^{\pm}(M,\pi)(x).$$ 

We first focus on the upper bound:
\begin{align*}
    N_I(\pi, x)\leq \pi_L(x) \hat{F}^+_M(0) + \sum_{m=1}^M     \hat{F}^+_M(m)  \sum_{\Nr(\pfrak)\leq x} U_m(\cos (\theta_\pi(\pfrak))).
\end{align*}
Since the first summand on the left-hand side is independent of $\pi$, in order to compute the first moment we need to estimate 
\begin{align*}
\sum_{m=1}^M     \hat{F}^+_M(m) \sum_{\Nr(\pfrak)\leq x}\frac{1}{\#\PKN} \sum_{\pi\in \PKN} U_m( \cos (\theta_\pi(\pfrak)).
\end{align*}
Using \eqref{trace-formula-LLW} we see that this is
\begin{align*}
&=\frac{1}{\#\PKN}\sum_{\substack{m=2\\ m \,\textrm{even}}}^M\hat{F}_M^+(m) \sum_{\Nr(\pfrak)\leq x} \left( C_L\Nr(\mathfrak{n})\prod_{i=1}^{d-1} \frac{k_i-1}{4\pi} 
  \Nr(\pfrak)^{-\frac{m}{2}}+O(  \Nr(\pfrak)^{\frac{3m}{2}}  \Nr(\mathfrak{n})^\epsilon)\right)\\
  & \quad + O_L\left( \sum_{m=1}^M\left|\hat{F}_M^+(m)\right| \sum_{\Nr(\pfrak)\leq x} \frac{ \Nr(\pfrak)^{\frac{3m}{2}}  \Nr(\mathfrak{n})^\epsilon }{\prod_{i=1}^{d-1} (k_i-1) }\right). 
\end{align*}

Using Proposition \ref{trace-formula} with $h=1$, the estimate $\hat{F}^+_M(m)\ll \frac{1}{m}$ and Lemma \ref{NP>=1}, we see that 
$$
\frac{1}{\#\PKN}\sum_{m=1}^M\hat{F}_M^+(m) \sum_{\Nr(\pfrak)\leq x}\sum_{\pi\in \PKN} U_m(\cos (\theta_\pi(\pfrak))\ll_L \log\log x+\frac{\pi_L(x) x^{\frac{3}{2}M}}{\prod_{i=1}^{d-1} k_i}.
$$
Keeping  \eqref{expected-value1} in mind we therefore get, 
$$
\frac{1}{\#\PKN} \sum_{\pi\in \PKN} N_I(\pi,x)-\pi_L(x) \mu_{\infty}(I) \ll_L \log\log x+ \frac{\pi_L(x) x^{\frac{3}{2}M}}{\prod_{i=1}^{d-1} k_i}+\frac{\pi_L(x)}{M+1} .
$$
On choosing $M=\left[\frac{2d\sum_{i=1}^{d-1} \log k_i}{3\log x}\right]$
for some $0<d<1$, we have proved the following proposition. 
\begin{proposition}
    Let $\sum_{i=1}^{d-1} \log k_i$ be a function of $x$. Then, for any $I\subseteq [0,\pi]$ we have $$
 \frac{1}{\#\PKN} \sum_{\pi\in \PKN} N_I(\pi,x) =\pi_L(x) \mu_{\infty}(I)+O_L\left( \frac{\pi_L (x)\log x}{\sum_{i=1}^{d-1} \log k_i}+\log\log x\right). 
$$

\end{proposition}
{\bf Remark:} The above proposition is the analogue of the effective average Sato-Tate theorem for holomorphic cusp forms \cite[Proposition 4.1]{Prabhu-Sinha} and Maass forms \cite[Theorem 1.1]{Wang}.

\section{Outline of the proof of the main theorem.}\label{Outline}
We compute all higher moments by proving the following result.
\begin{thm} \label{higher-moments-theorem} Let $M = \lfloor \sqrt{\pi_L(x)}\log\log x\rfloor$ and $F^{\pm}(M,\pi)(x)$ be as defined in \eqref{F(M,pi)-defn}. Suppose $\frac{\sum_{i=1}^d \log k_i}{\sqrt{x}\log x} \to \infty$ as $x\to\infty$. Then,
    \begin{equation}
    \lim_{x\to\infty} \left\langle \left(\frac{F^{\pm}(M, \pi)(x))}{\sqrt{\pi_L(x)}}\right)^n\right\rangle  = \begin{cases}
        0 &\text{ if } n \text{ is odd,}\\
        \frac{n!}{2^{\frac{n}{2}}\frac{n}{2}!}\left( \mu_\infty(I) - \mu_\infty(I)^2 \right)^\frac{n}{2} &\text{ if } n \text{ is even.}
    \end{cases}
\end{equation}
\end{thm}
Following the strategy in \cite[Section 6]{Prabhu-Sinha}, and making the necessary modifications i.e., using the trace formula in Proposition \ref{trace-formula-LLW} instead of the Eichler-Selberg trace formula, one can prove the following analogue of \cite[Proposition 6.2]{Prabhu-Sinha}. The proof is very similar, so we omit details. 
\begin{proposition}
    Let $I= [a,b]\subseteq [0,\pi]$ and $M = \lfloor \sqrt{\pi_L(x)} \log\log x\rfloor$. Suppose $\frac{\sum_{i=1}^d \log k_i}{\sqrt{x}\log x} \to \infty$ as $x\to\infty$,
    $$ \lim\limits_{x\to\infty} \left\langle \left| \frac{N_I(\pi, x) -\pi_L(x)\mu_\infty(I) - F^{\pm}(M,\pi)(x)}{\sqrt{ \pi_L(x)(\mu_\infty(I) - \mu_\infty(I)^2)}} \right|^2 \right\rangle = 0.$$
\end{proposition}
This implies that under the conditions of the above proposition, the quantity $$\frac{F^{\pm}(M,\pi)(x)}{\sqrt{ \pi_L(x)(\mu_\infty(I) - \mu_\infty(I)^2)}}$$ converges in mean square to 
\begin{equation}\label{Normalized RV}
    \frac{N_I(\pi, x) -\pi_L(x)\mu_\infty(I) }{\sqrt{ \pi_L(x)(\mu_\infty(I) - \mu_\infty(I)^2)}}
\end{equation} as $x\to\infty$. 
Since convergence in mean square implies convergence in distribution (see \cite[Chapter 6, Theorems 5 and 7]{Rohtagi-book}), and the normal distribution is characterized by its moments, Theorem \ref{higher-moments-theorem} immediately gives us that the quantity in \eqref{Normalized RV} follows the Gaussian distribution. This completes the proof of Theorem \ref{main-thm}.
\section{Higher Moments}\label{higher-moments}
In this section we give the details of the proof of Theorem \ref{higher-moments-theorem}. In order to simplify calculations while computing higher moments, it is useful to express the main term of the trace formula \ref{trace-formula-LLW} using an integral. The following lemma is Proposition 29.11 in \cite{Knightly-Li}, with the prime $p$ replaced with $\Nr(\pfrak)$. We prove it here for completeness.
\begin{lemma}\label{key-lemma}
    For a non-archimedean valuation $\nu$ corresponding to a prime ideal $\pfrak$, we define
    \begin{equation}
        d\mu_\nu (\theta) = \frac{ \Nr(\pfrak) + 1}{\left( \Nr(\pfrak)^{\frac{1}{2}} + \Nr(\pfrak)^{-\frac{1}{2}}\right)^2 -4\cos^2\theta}d\mu_\infty(\theta),
    \end{equation}
    where $ d\mu_\infty(\theta) = \frac{2}{\pi} \sin^2\theta ~d\theta.$ Then,
    \begin{equation}\label{alt-trace-integral}
        \int_0^{\pi} U_{m}(\cos\theta)d\mu_\nu(\theta) = \begin{cases}
            \Nr(\pfrak)^{-\frac{m}{2}} & \text{ if $m$ is even}\\
            0 & \text{ if $m$ is odd}.
        \end{cases}
    \end{equation}
\end{lemma}
\begin{proof}
    For $|t|<1$ and $x\in [-1,1]$, the generating function of the Chebyshev polynomials of the second kind is given by 
    \begin{equation}
        \sum_{n=0}^\infty U_n(x)t^n = \frac{1}{1-2xt +t^2}.
    \end{equation} Substituting $t= \pm \Nr(\pfrak)^{-\frac{1}{2}}$ and adding, we obtain
    $$\sum_{n=0}^\infty U_{2n}(x)\Nr(\pfrak)^{-n} = \frac{ \Nr(\pfrak) + 1}{\left( \Nr(\pfrak)^{\frac{1}{2}} + \Nr(\pfrak)^{-\frac{1}{2}}\right)^2 -4x^2}. $$
    Denoting the above quantity by $u_\pfrak(x)$, and letting $x=\cos\theta$, we see that 
    \begin{align*}
         \int_0^{\pi} U_{m}(\cos\theta)d\mu_\nu(\theta) &= \int_0^{\pi} U_{m}(\cos\theta) \left( \sum_{n=0}^\infty U_{2n}(x)\Nr(\pfrak)^{-n}\right) d\mu_\infty(\theta)\\
         &= \begin{cases}
            \Nr(\pfrak)^{-\frac{m}{2}} & \text{ if $m$ is even}\\
            0 & \text{ if $m$ is odd},
            \end{cases}
    \end{align*}
    using \eqref{orthogonality-rel'ns}, the orthogonality of the Chebyshev polynomials.
\end{proof}

As a consequence of Lemma \ref{key-lemma}, the following alternate expression for the trace formula \eqref{trace-formula-LLW} holds.
\begin{lemma}\label{alternate version of trace formula}
Let $\underline{m}=(m_1,\ldots,m_h)$ be a tuple of non-negative integers. Further, let $\pfrak_1,\ldots, \pfrak_h$ 
 be distinct prime ideals coprime to the squarefree ideal $\mathfrak{n}.$ Then, 
\begin{equation}\label{alt-trace-formula}
    \frac{1}{\#\PKN}\sum_{\pi \in \PKN} \prod_{i=1}^h U_{m_i}(\cos\theta_\pi(\pfrak_i)) = \prod_{i=1}^h \int_0^\pi U_{m_i}(\cos\theta)d\mu_{\nu_i}(\theta) + O_{L, \mathfrak{n}}\left( \frac{\prod_{i=1}^h \Nr(\pfrak_i)^{\frac{3m_i}{2}}}{\prod_{i=1}^d k_i} \right). 
\end{equation}   
\end{lemma}
\begin{proof}
    This follows easily by using \eqref{alt-trace-integral} in the main term of trace formula in \eqref{trace-formula-LLW} and noting \eqref{size of PKN}.
\end{proof}
We now proceed to calculate the higher moments $\left\langle \left(\frac{F^{\pm}(M, \pi)(x))}{\sqrt{\pi_L(x)}}\right)^n\right\rangle$.
Observe that 
\begin{align}\label{Highermoment eqn}
    \left\langle \left(\frac{F^{\pm}(M, \pi)(x))}{\sqrt{\pi_L(x)}}\right)^n\right\rangle &=
    \frac{1}{\#\PKN}\frac{1}{\pi_L(x)^{\frac{n}{2}}} \sum_{\pi \in \PKN} \left( \sum_{\Nr(\pfrak)\leq x} \sum_{m=1}^M \hat{F}^{\pm}_M(m)U_m(\cos\theta_\pi(\pfrak))\right)^n
\end{align}
Let $Z^{\pm}_M( \theta) = \sum_{m=1}^M \hat{F}_M^{\pm}(m)U_m(\cos\theta)$. 
Then 
\begin{align*}
    \left( \sum_{\Nr(\pfrak)\leq x} \sum_{m=1}^M \hat{F}^{\pm}_M(m)U_m(\cos\theta_\pi(\pfrak))\right)^n &= \left( \sum_{\Nr(\pfrak)\leq x} Z_M^{\pm}(\theta_\pi(\pfrak))\right)^n
    \\
    &= \sum_{u=1}^n\sum_{(r_1,\ldots,r_u)}^{(1)} \frac{n!}{r_1!\cdots r_u!}\frac{1}{u} \sum_{(\pfrak_1, \ldots, \pfrak_u)}^{(2)} \prod_{i=1}^u Z_M^{\pm}(\theta_\pi(\pfrak_i))^{r_i},
\end{align*}
where 
\begin{itemize}
    \item The sum $\sum\limits_{(r_1,\ldots,r_u)}^{(1)}$ is taken over tuples of positive integers $(r_1, \ldots, r_u)$ such that $r_1+\cdots+r_u=n$.
    \item The sum $\sum\limits_{(\pfrak_1, \ldots, \pfrak_u)}^{(2)}$ is taken over tuples of distinct prime ideals $\pfrak_1, \ldots, \pfrak_u$ coprime to $\mathfrak{n}$ with $\Nr(\pfrak_i)\leq x$ for each $i=1,\ldots, u$.
\end{itemize}

Let $[f(\theta)\bullet U_n(\cos\theta)]$ denote the $n$-th Fourier coefficient when $f$ is expanded as a Fourier series using the orthogonal basis of Chebyshev polynomials of the second kind. Explicitly, 
$$ [f(\theta)\bullet U_n(\cos\theta)] = \frac{2}{\pi} \int_0^\pi f(\theta)\, U_n(\cos\theta) \sin^2\theta ~d\theta.$$Expanding  each $Z_M^{\pm}(\theta_\pi(\pfrak_i))^{r_i}$ in this way we have
\begin{align*}
    \prod_{i=1}^u Z_M^{\pm}(\theta_\pi(\pfrak_i))^{r_i} &= \prod_{i=1}^u \left( \sum_{m_i=1}^{Mr_i} \left[Z^\pm_M(\theta)^{r_i} \bullet U_{m_i}(\cos\theta)\right] 
          U_{m_i}(\cos\theta_\pi(\pfrak_i))\right)\\
    &= \sum_{(m_1,\ldots, m_u)}^{(3)}\prod_{i=1}^u \left[Z^\pm_M(\theta)^{r_i} \bullet U_{m_i}(\cos\theta)\right]   U_{m_i}(\cos\theta_\pi(\pfrak_i)).
\end{align*}
Here, $\sum\limits_{(m_1,\ldots, m_u)}^{(3)}$ denote that the sum is taken over tuples where each $m_i$ ranges from $1$ to $Mr_i$. Averaging over $\pi \in \PKN$, and using the trace formula given in equation \eqref{alt-trace-formula}, 
\begin{align*}
    &\frac{1}{\#\PKN} \sum_{\pi\in\PKN}  \prod_{i=1}^u Z_M^{\pm}(\theta_\pi(\pfrak_i))^{r_i} \\
    &= \sum_{(m_1,\ldots, m_u)}^{(3)}   \prod_{i=1}^u \left[Z^\pm_M(\theta)^{r_i} \bullet U_{m_i}(\cos\theta)\right]   \frac{1}{\#\PKN} \sum_{\pi\in\PKN} \prod_{i=1}^u U_{m_i}(\cos\theta_\pi(\pfrak_i))\\
    &= \sum_{(m_1,\ldots, m_u)}^{(3)}   \prod_{i=1}^u 
 \left[Z^\pm_M(\theta)^{r_i} \bullet U_{m_i}(\cos\theta)\right]   \left( \prod_{i=1}^u \int_0^\pi U_{m_i}(\cos\theta)d\mu_{\nu_i}(\theta) + O_{L, \mathfrak{n}}\left( \frac{\prod_{i=1}^u \Nr(\pfrak_i)^{\frac{3}{2}m_i}}{\prod_{i=1}^d k_i} \right)   \right) \\
    &= \prod_{i=1}^u \int_0^\pi Z_M^{\pm}(\theta)^{r_i}~d\mu_{\nu_i}(\theta) + O_{L,\mathfrak{n}}\left(  \frac{1}{\prod_{i=1}^d k_i} \prod_{i=1}^u \sum_{m_i=1}^{Mr_i}  \langle Z_M^{\pm}( \cdot )^{r_i}, U_{m_i}(\cos 
 \cdot)\rangle \Nr(\pfrak_i)^{\frac{3}{2}m_i} \right).
\end{align*}
Let us estimate the error term. Observe that $Z_M^{\pm}(\theta) = F_M^{\pm}(\theta) - \hat{F}_M^{\pm}(0)$, so it is absolutely bounded. Therefore, $[Z_M^{\pm}(\theta)^{r_i}\bullet  U_{m_i}(\cos\theta)] \ll m_i$ for each $i=1,\ldots, u$ using the trivial bound $ |U_n(\cos\theta)|\leq n+1.$ We also know that $\Nr(\pfrak_i)\leq x$, so 
\begin{equation*}
    \prod_{i=1}^u \sum_{m_i=1}^{Mr_i} \left[Z^\pm_M(\theta)^{r_i} \bullet U_{m_i}(\cos\theta)\right]   \Nr(\pfrak_i)^{\frac{3}{2}m_i} \ll_n M^{2u}x^{\frac{3}{2}Mu}.
\end{equation*}
Therefore, the higher moments
$$\left\langle \left(\frac{F^{\pm}(M, \pi)(x))}{\sqrt{\pi_L(x)}}\right)^n\right\rangle $$
are given by 
\begin{align}
  \label{higher-moments-integral}  & \frac{1}{\pi_L(x)^{\frac{n}{2}}}\sum_{u=1}^n\sum_{(r_1,\ldots,r_u)}^{(1)} \frac{n!}{r_1!\cdots r_u!}\frac{1}{u} \sum_{(\pfrak_1, \ldots, \pfrak_u)}^{(2)} \prod_{i=1}^u \int_0^\pi Z_M^{\pm}(\theta)^{r_i}~d\mu_\nu(\theta) + O_{L,n}\left( \frac{M^{2n}x^{\frac{3}{2}Mn}\pi_L(x)^{\frac{n}{2}}}{\prod_{i=1}^d k_i}\right). 
\end{align}
We now analyze the main term in the above equation. First, we prove:
\begin{lemma}
    Assume the notations introduced earlier. Then, the following hold. 
 \begin{align}
     \label{connecting mu_p to mu_infty}
    \int_0^\pi Z_M^{\pm}(\theta)^{r}~d\mu_\nu(\theta) & = \int_0^\pi Z_M^{\pm}(\theta)^{r} ~d\mu_\infty(\theta) + O\left( \frac{1}{\Nr(\pfrak)}\right)\\
    \label{int-Z_M} \int_0^\pi Z_M^{\pm}(\theta)~d\mu_\nu(\theta) &= O\left( \frac{1}{\Nr(\pfrak)}\right)\\
    \label{int-Z_M^r} \int_0^\pi Z_M^{\pm}(\theta)^{r}~d\mu_\nu(\theta) & =\begin{cases}
        \sum_{m=1}^M \hat{F}_M^{\pm}(m)^2 + O\left( \frac{1}{\Nr(\pfrak)}\right) & \text{ if } r=2\\
        O(1) &\text{ for } r>2.
    \end{cases}
 \end{align}
\end{lemma}
\begin{proof}
    Equation \eqref{connecting mu_p to mu_infty} follows from noting that $Z_M^{\pm}(\theta)$ is bounded and that for any prime ideal $\pfrak$,
\begin{align*}
    \frac{2}{\pi}\frac{ (\Nr(\pfrak) + 1)\sin^2\theta}{\left( \Nr(\pfrak)^{\frac{1}{2}} + \Nr(\pfrak)^{-\frac{1}{2}}\right)^2 -4\cos^2\theta} &= \frac{2}{\pi}\sin^2\theta + O\left( \frac{1}{\Nr(\pfrak)}\right).
\end{align*}
Equation \eqref{int-Z_M} follows from letting $r=1$ in \eqref{connecting mu_p to mu_infty} and the orthogonality of Chebyshev polynomials $U_m(\cos\theta)$. 
Finally, 
\begin{align*}
    \int_0^\pi Z_M^{\pm}(\theta)^{2}~d\mu_\nu(\theta) &= \int_0^\pi Z_M^{\pm}(\theta)^{2} ~d\mu_\infty(\theta) + O\left( \frac{1}{\Nr(\pfrak)}\right) \quad \text{ using } \eqref{connecting mu_p to mu_infty}\\
    &= \int_0^\pi \left( \sum_{m=1}^M \hat{F}_M^{\pm}(m)U_m(\cos\theta)\right)^2 ~d\mu_\infty(\theta) + O\left( \frac{1}{\Nr(\pfrak)}\right) \\
    &= \int_0^\pi \sum_{m_1, m_2 =1}^M \hat{F}_M^{\pm}(m_1)\hat{F}_M^{\pm}(m_2 )\sum_{k= 0}^{\min\{m_1, m_2\}} U_{m_1+m_2-2k}(\cos\theta) ~d\mu_\infty(\theta) + O\left( \frac{1}{\Nr(\pfrak)}\right)\\
    &= \sum_{m=1}^M \hat{F}_M^{\pm}(m)^2 + O\left( \frac{1}{\Nr(\pfrak)}\right)
\end{align*}
using orthogonality. If $r>2$, the claimed estimate follows because $Z^{\pm}_M$ is bounded. This proves \eqref{int-Z_M^r}.
\end{proof}

We now have the tools to work out the integrals in \eqref{higher-moments-integral}. We do so by dividing the partitions into three types.\\
{\bf Case 1:} If $(r_1,\ldots,r_u)= (2,\ldots,2)$, i.e., each part is equal to $2$, then $n$ is even and $u=n/2$. In this case, the sum corresponding to this partition is
\begin{align*}
    &\frac{1}{\pi_L(x)^{\frac{n}{2}}} \frac{n!}{2^{\frac{n}{2}}\frac{n}{2}!} \sum_{(\pfrak_1, \ldots, \pfrak_u)}^{(2)} \prod_{i=1}^{n/2} \int_0^\pi Z_M^{\pm}(\theta)^{2}~d\mu_\nu(\theta)\\
    &= \frac{1}{\pi_L(x)^{\frac{n}{2}}} \frac{n!}{2^{\frac{n}{2}}\frac{n}{2}!} \sum_{(\pfrak_1, \ldots, \pfrak_u)}^{(2)} \prod_{i=1}^{n/2} \left( \sum_{m=1}^M \hat{F}_M^{\pm}(m)^2 + O\left( \frac{1}{\Nr(\pfrak_i)}\right)\right)\\
    &= \frac{n!}{2^{\frac{n}{2}}\frac{n}{2}!}\left( \sum_{m=1}^M \hat{F}_M^{\pm}(m)^2\right)^\frac{n}{2} + o(\pi_L(x)).
\end{align*}

Using equation \eqref{variance} and noting that $M$ is an increasing function of $x$, we conclude 
\begin{equation}\label{case-1-limit}
    \lim\limits_{x\to\infty} \frac{1}{\pi_L(x)^{\frac{n}{2}}} \frac{n!}{2^{\frac{n}{2}}\frac{n}{2}!} \sum_{(\pfrak_1, \ldots, \pfrak_u)}^{(2)} \prod_{i=1}^{n/2} \int_0^\pi Z_M^{\pm}(\theta)^{2}~d\mu_\nu(\theta) = \frac{n!}{2^{\frac{n}{2}}\frac{n}{2}!} (\mu_\infty(I) - \mu_\infty(I)^2)^{\frac{n}{2}}.
\end{equation}

\noindent
{\bf Case 2:} If $(r_1, \ldots, r_u)$ has $\ell$ parts equal to $1$ with $1\leq \ell\leq n$. Then $$n= r_1 +\cdots +r_u \geq \ell + 2(u-\ell).$$ Therefore in this case, $$ u-\ell \leq \frac{n-\ell}{2} \leq \frac{n-1}{2}.$$ Using equations \eqref{int-Z_M} and \eqref{int-Z_M^r}, 
\begin{align*}
    &\frac{1}{\pi_L(x)^{\frac{n}{2}}} \frac{n!}{2^{\frac{n}{2}}\frac{n}{2}!} \sum_{(\pfrak_1, \ldots, \pfrak_u)}^{(2)} \prod_{i=1}^{u} \int_0^\pi Z_M^{\pm}(\theta)^{r_i}~d\mu_\nu(\theta)\\
    &\ll \frac{1}{\pi_L(x)^{\frac{n}{2}}} (\log\log x)^{\ell} \pi_L(x)^{(u-\ell)}\\
    &\ll \pi_L(x)^{-\frac{1}{2}}(\log\log x)^{\ell}.
\end{align*} 
{\bf Case 3:} The remaining case is where $(r_1, \ldots, r_u)$ has all parts $r_i\geq 2$ and at least one part greater than or equal to $3$. In this case, it is easy to see that $u\leq \frac{n}{2}-1$. So we have 
\begin{align*}
    &\frac{1}{\pi_L(x)^{\frac{n}{2}}} \frac{n!}{2^{\frac{n}{2}}\frac{n}{2}!} \sum_{(\pfrak_1, \ldots, \pfrak_u)}^{(2)} \prod_{i=1}^{u} \int_0^\pi Z_M^{\pm}(\theta)^{r_i}~d\mu_\nu(\theta)\\
    &\ll \pi_L(x)^{-1},
\end{align*}
using \eqref{int-Z_M^r} for each part. 
Therefore, for partitions $(r_1,\ldots, r_u)$ described in Case 2 and Case 3, 
\begin{equation}\label{case2,3-limit}
    \lim\limits_{x\to\infty} \frac{1}{\pi_L(x)^{\frac{n}{2}}} \frac{n!}{2^{\frac{n}{2}}\frac{n}{2}!} \sum_{(\pfrak_1, \ldots, \pfrak_u)}^{(2)} \prod_{i=1}^{n/2} \int_0^\pi Z_M^{\pm}(\theta)^{r_i}~d\mu_\nu(\theta) = 0.
\end{equation}

Gathering equations \eqref{case-1-limit}, \eqref{case2,3-limit} and choosing $M = \lfloor \sqrt{\pi_L(x)}\log\log x \rfloor$, we have proved
\begin{equation*}
    \lim_{x\to\infty} \left\langle \left(\frac{F^{\pm}(M, \pi)(x))}{\sqrt{\pi_L(x)}}\right)^n\right\rangle  = \begin{cases}
        0 &\text{ if } n \text{ is odd,}\\
        \frac{n!}{2^{\frac{n}{2}}\frac{n}{2}!}\left( \mu_\infty(I) - \mu_\infty(I)^2\right)^\frac{n}{2} &\text{ if } n \text{ is even.}
    \end{cases}
\end{equation*}
This completes the proof. 

\section{Improvements}
We can obtain an improvement on the growth conditions on the weight vector $\underline{k}= \underline{k}(x)$ in Theorem \ref{higher-moments-theorem}, if the indicator function $\chi_I$ were to be replaced by a smooth test function, adapting the line of proof in \cite[Theorem 1.6]{Baier-Prabhu-Sinha}. As in Theorem \ref{main-thm}, $\underline{k}(x)$ is a vector $(k_1(x), \ldots, k_d(x))$ where the components run over even integers $\geq 4$. More precisely, the following holds. 

\begin{thm}\label{main-elliptic}
 Let $\Phi \in C^{\infty}(\R)$ be a real-valued, even function in the Schwartz class and $\widehat{\Phi}$ its Fourier transform. 
 Fix a real number
 $M\ge 1$ and define 
$$\phi_M(t) = \sum_{m \in \Z} \Phi(M(t+m)) \text{ and } V_{\Phi,M} = \int_0^1\phi_M(t)^2\mu_\infty(t)dt - \left( \int_0^1 \phi_M(t) \mu_\infty(t) dt \right)^2.$$
 For $\pi \in \PKN,$ define
$$N_{\Phi,M,\pi}(x) = \sum_{\substack{\Nr(\pfrak) \leq x \\ \pfrak \nmid \mathfrak{n}}} \phi_M(\theta_\pi(\pfrak)).$$
\begin{enumerate}
\item[{\bf(a)}] Suppose $\widehat{\Phi}$ is compactly supported and $\underline{k} = \underline{k}(x)$ satisfies $\frac{\sum_{i=1}^d\log k_i}{\log x} \to \infty$ as $x \to \infty.$  Then,
for any integer $r \geq 0,$
\begin{equation}\label{Gaussian}
\lim_{x \to \infty}\frac{1}{|\PKN|}\sum_{\pi \in \PKN} \left(\frac{N_{\Phi,M,\pi}(x) - {\pi}_L(x)\int_0^1 \phi_M(t) \mu_\infty(t)dt}{\sqrt{{\pi}(x)V_{\Phi,M}}}\right)^r = 
\begin{cases}
0 &\text{ if }r\text{ is odd}\\
\frac{r!}{\left(\frac{r}{2}\right)!2^{r/2}}&\text{ if }r\text{ is even.}
\end{cases}
\end{equation}

\item[{\bf(b)}] For fixed $\lambda, \omega>0$, suppose the Fourier transform $\widehat{\Phi}$ satisfies $\widehat{\Phi}(t) \ll e^{-\lambda |t|^{\omega}}$, as $|t| \to \infty.$  Then, the asymptotic \eqref{Gaussian} holds if $\underline{k} = \underline{k}(x)$ satisfies $\frac{\sum_{i=1}^d\log k_i}{(\log x)^{1+1/\omega}} \to \infty$ as $x \to \infty.$
\end{enumerate}
\end{thm}

\bibliographystyle{alpha}
\bibliography{Sato-TateNT.bib}

\begin{thebibliography}{HIJTS22}

\bibitem[BJ79]{BorelJacquet}
A.~Borel and H.~Jacquet.
\newblock Automorphic forms and automorphic representations.
\newblock In {\em Automorphic forms, representations and {$L$}-functions
  ({P}roc. {S}ympos. {P}ure {M}ath., {O}regon {S}tate {U}niv., {C}orvallis,
  {O}re., 1977), {P}art 1}, Proc. Sympos. Pure Math., XXXIII, pages 189--207.
  Amer. Math. Soc., Providence, R.I., 1979.
\newblock With a supplement ``On the notion of an automorphic representation''
  by R. P. Langlands.

\bibitem[BLGG11]{BLGG}
Thomas Barnet-Lamb, Toby Gee, and David Geraghty.
\newblock The {S}ato-{T}ate conjecture for {H}ilbert modular forms.
\newblock {\em J. Amer. Math. Soc.}, 24(2):411--469, 2011.

\bibitem[BP19]{Baier-Prabhu}
Stephan Baier and Neha Prabhu.
\newblock Moments of the error term in the {S}ato-{T}ate law for elliptic
  curves.
\newblock {\em J. Number Theory}, 194:44--82, 2019.

\bibitem[BPS20]{Baier-Prabhu-Sinha}
Stephan Baier, Neha Prabhu, and Kaneenika Sinha.
\newblock Central limit theorems for elliptic curves and modular forms with
  smooth weight functions.
\newblock {\em J. Math. Anal. Appl.}, 485(1):29, 2020.
\newblock Article no. 123709.

\bibitem[BS19]{BB-KS}
Baskar Balasubramanyam and Kaneenika Sinha.
\newblock Pair correlation statistics for {S}ato-{T}ate sequences.
\newblock {\em J. Number Theory}, 202:107--140, 2019.

\bibitem[CN63]{CN}
K.~Chandrasekharan and Raghavan Narasimhan.
\newblock The approximate functional equation for a class of zeta-functions.
\newblock {\em Math. Ann.}, 152:30--64, 1963.

\bibitem[FR10]{Faifman-Rudnick}
Dmitry Faifman and Ze\'{e}v Rudnick.
\newblock Statistics of the zeros of zeta functions in families of
  hyperelliptic curves over a finite field.
\newblock {\em Compos. Math.}, 146(1):81--101, 2010.

\bibitem[FW21]{Fugleberg-Walji}
Nathan Fugleberg and Nahid Walji.
\newblock On the distribution of traces of frobenius for families of elliptic
  curves and the lang-trotter conjecture on average, 2021.

\bibitem[HIJTS22]{HIJT}
Alexandra Hoey, Jonas Iskander, Steven Jin, and Fernando Trejos~Su\'{a}rez.
\newblock An unconditional explicit bound on the error term in the
  {S}ato-{T}ate conjecture.
\newblock {\em Q. J. Math.}, 73(4):1189--1225, 2022.

\bibitem[KL06]{Knightly-Li}
Andrew Knightly and Charles Li.
\newblock {\em Traces of {H}ecke operators}, volume 133 of {\em Mathematical
  Surveys and Monographs}.
\newblock American Mathematical Society, Providence, RI, 2006.

\bibitem[Li09]{Li}
Charles Li.
\newblock On the distribution of {S}atake parameters of {${\rm GL}_2$}
  holomorphic cuspidal representations.
\newblock {\em Israel J. Math.}, 169:341--373, 2009.

\bibitem[LLW14]{Lau-Li-Wang}
Yuk-Kam Lau, Charles Li, and Yingnan Wang.
\newblock Quantitative analysis of the {S}atake parameters of {${\rm GL}_2$}
  representations with prescribed local representations.
\newblock {\em Acta Arith.}, 164(4):355--380, 2014.

\bibitem[LNW19]{lau-Ng-Wang}
Yuk-Kam Lau, Ming~Ho Ng, and Yingnan Wang.
\newblock Statistics of {H}ecke eigenvalues for {${\rm GL}(n)$}.
\newblock {\em Forum Math.}, 31(1):167--185, 2019.

\bibitem[Mon94]{Montgomery}
Hugh~L. Montgomery.
\newblock {\em Ten lectures on the interface between analytic number theory and
  harmonic analysis}, volume~84 of {\em CBMS Regional Conference Series in
  Mathematics}.
\newblock Published for the Conference Board of the Mathematical Sciences,
  Washington, DC; by the American Mathematical Society, Providence, RI, 1994.

\bibitem[MP20]{Murty-Prabhu}
M.~Ram Murty and Neha Prabhu.
\newblock Central limit theorems for sums of quadratic characters, {H}ecke
  eigenforms, and elliptic curves.
\newblock {\em Proc. Amer. Math. Soc.}, 148(3):965--977, 2020.

\bibitem[Nag06]{Nagoshi}
Hirofumi Nagoshi.
\newblock Distribution of {H}ecke eigenvalues.
\newblock {\em Proc. Amer. Math. Soc.}, 134(11):3097--3106, 2006.

\bibitem[Pra17]{NPThesis}
Neha Prabhu.
\newblock {\em Fluctuations in the distribution of Hecke eigenvalues}.
\newblock IISER Pune, 2017.
\newblock Thesis (Ph.D.).

\bibitem[PS19]{Prabhu-Sinha}
Neha Prabhu and Kaneenika Sinha.
\newblock Fluctuations in the distribution of {H}ecke eigenvalues about the
  {S}ato-{T}ate measure.
\newblock {\em Int. Math. Res. Not. IMRN}, (12):3768--3811, 2019.

\bibitem[RES01]{Rohtagi-book}
Vijay~K. Rohatgi and A.~K.~Md. Ehsanes~Saleh.
\newblock {\em An introduction to probability and statistics}.
\newblock Wiley Series in Probability and Statistics: Texts and References
  Section. Wiley-Interscience, New York, second edition, 2001.

\bibitem[Ros99]{MR}
Michael Rosen.
\newblock A generalization of {M}ertens' theorem.
\newblock {\em J. Ramanujan Math. Soc.}, 14(1):1--19, 1999.

\bibitem[RT17]{RT2017}
Jeremy Rouse and Jesse Thorner.
\newblock The explicit {S}ato-{T}ate conjecture and densities pertaining to
  {L}ehmer-type questions.
\newblock {\em Trans. Amer. Math. Soc.}, 369(5):3575--3604, 2017.

\bibitem[SWZ]{Sun-Wen-Zhang}
Lehan Sun, Y~Wen, and X~Zhang.
\newblock Remark on the paper “fluctuations in the distribution of hecke
  eigenvalues about the sato-tate measure”.
\newblock {\em (unpublished)}.

\bibitem[Tay08]{Taylor2008}
Richard Taylor.
\newblock Automorphy for some {$l$}-adic lifts of automorphic mod {$l$}
  {G}alois representations. {II}.
\newblock {\em Publ. Math. Inst. Hautes \'Etudes Sci.}, (108):183--239, 2008.

\bibitem[Tho21]{Thorner21}
Jesse Thorner.
\newblock Effective forms of the {S}ato-{T}ate conjecture.
\newblock {\em Res. Math. Sci.}, 8(1):Paper No. 4, 21, 2021.

\bibitem[Wan14]{Wang}
Yingnan Wang.
\newblock The quantitative distribution of {H}ecke eigenvalues.
\newblock {\em Bull. Aust. Math. Soc.}, 90(1):28--36, 2014.

\end{thebibliography}

\end{document}